\documentclass[11pt]{amsart}%
\usepackage{palatino, mathpazo}
\usepackage{amsfonts}
\usepackage{amsmath}
\usepackage{amssymb,latexsym}
\usepackage{graphicx}
\usepackage{amssymb}
\usepackage[mathscr]{eucal}
\usepackage{color}%
\providecommand{\U}[1]{\protect \rule{.1in}{.1in}}
\newtheorem{theorem}{Theorem}[section]

\theoremstyle{remark}
\newtheorem{remark}[theorem]{Remark}

\numberwithin{equation}{section}

\begin{document}
\title[reducible monodromy representation]{On reducible monodromy representations of some generalized Lam\'{e} equation}
\author{Zhijie Chen}
\address{Department of Mathematical Sciences and Yau Mathematical Sciences Center,
Beijing, 100084, China }
\email{zjchen@math.tsinghua.edu.cn}
\author{Ting-Jung Kuo}
\address{Taida Institute for Mathematical Sciences (TIMS), National Taiwan University,
Taipei 10617, Taiwan }
\email{tjkuo1215@gmail.com}
\author{Chang-Shou Lin}
\address{Taida Institute for Mathematical Sciences (TIMS), Center for Advanced Study in
Theoretical Sciences (CASTS), National Taiwan University, Taipei 10617, Taiwan }
\email{cslin@math.ntu.edu.tw}
\author{Kouichi Takemura}
\address{School of Mathematics, University of Leeds, Leeds LS2 9JT, United Kingdom\\
Department of Mathematics, Faculty of Science and Engineering, Chuo
University, 1-13-27 Kasuga, Bunkyo-ku Tokyo 112-8551, Japan}
\email{takemura@math.chuo-u.ac.jp}

\begin{abstract}
In this note, we compute the explicit formula of the monodromy data for a
generalized Lam\'{e} equation when its monodromy is reducible but not
completely reducible. We also solve the corresponding Riemman-Hilbert problem.

\end{abstract}
\maketitle

\section{Introduction}

Throughout the paper, we use the notations $\omega_{0}=0$, $\omega_{1}=1$,
$\omega_{2}=\tau$, $\omega_{3}=1+\tau$ and $\Lambda_{\tau}=\mathbb{Z+Z}\tau$,
where $\tau \in \mathbb{H}=\{ \tau|\operatorname{Im}\tau>0\}$. Define $E_{\tau
}:=\mathbb{C}/\Lambda_{\tau}$ to be a flat torus in the plane and $E_{\tau
}[2]:=\{ \frac{\omega_{k}}{2}|0\leq k\leq3\}+\Lambda_{\tau}$ to be the set
consisting of the lattice points and 2-torsion points in $E_{\tau}$.

Let $\wp(z)=\wp(z|\tau)$ be the Weierstrass elliptic function with periods
$\Lambda_{\tau}$, defined by%
\[
\wp(z|\tau):=\frac{1}{z^{2}}+\sum_{\omega \in \Lambda_{\tau}\backslash
\{0\}}\left(  \frac{1}{(z-\omega)^{2}}-\frac{1}{\omega^{2}}\right)  ,
\]
and $e_{k}=e_{k}(\tau):=\wp(\frac{\omega_{k}}{2}|\tau)$, $k\in \{1,2,3\}$. Let
$\zeta(z)=\zeta(z|\tau):=-\int^{z}\wp(\xi|\tau)d\xi$ be the Weierstrass zeta
function, which is an odd meromorphic function with two quasi-periods:%
\begin{equation}
\eta_{1}(\tau)=\zeta(z+1|\tau)-\zeta(z|\tau),\text{ \ }\eta_{2}(\tau
)=\zeta(z+\tau|\tau)-\zeta(z|\tau). \label{quasi}%
\end{equation}

In this note, we study the following generalized Lam\'{e} equation (GLE):%
\begin{equation}
y^{\prime \prime}(z)=\left[
\begin{array}
[c]{l}%
\sum_{k=0}^{3}n_{k}(n_{k}+1)\wp(z+\tfrac{\omega_{k}}{2})+\tfrac{3}{4}%
(\wp(z+p)\\
+\wp(z-p))+A(\zeta(z+p)-\zeta(z-p))+B
\end{array}
\right]  y(z)\text{ \ in }E_{\tau}, \label{GLE-2}%
\end{equation}
where $n_{k}\in \mathbb{N}\cup \{0\}$ for all $k$, $A,B\in \mathbb{C}$ and $\pm
p\not \in E_{\tau}[2]$ are always assumed to be \emph{apparent singularities}
(i.e. non-logarithmic). Under this assumption, $B$ is determined by $(p,A)$ as
follows (see \cite{Chen-Kuo-Lin}):%
\begin{equation}
B=A^{2}-\zeta(2p)A-\tfrac{3}{4}\wp(2p)-\sum_{k=0}^{3}n_{k}(n_{k}+1)\wp \left(
p+\tfrac{\omega_{k}}{2}\right)  . \label{i60}%
\end{equation}

We are interested in GLE (\ref{GLE-2}), which was studied in
\cite{Takemura,Chen-Kuo-Lin}, because it has a deep relation with the
well-known Panlev\'{e} VI equation. Indeed, if $(A(\tau),B(\tau),p(\tau))$
depends on $\tau$ suitably such that GLE (\ref{GLE-2}) preserves the monodromy
as $\tau$ deforms, then $p(\tau)$ satisfies the elliptic form of Panlev\'{e}
VI equation. See \cite{Chen-Kuo-Lin} or Section 3. Note that by letting
$x=\wp(z)$, GLE (\ref{GLE-2}) can be projected to a new equation on
$\mathbb{CP}^{1}$, which is a second order Fuchsian equation with five
singular points $\{e_{1},e_{2},e_{3},\wp(p),\infty \}$ with $\wp(p)$ being
apparent. Such type of ODEs on $\mathbb{CP}^{1}$ have been widely studied in
the literature. However, the monodromy of this new ODE are not easy to
compute. Therefore, it is more convenient for us to study GLE (\ref{GLE-2}) in
$E_{\tau}$ directly as long as the monodromy is concerned.

The monodromy representation of GLE (\ref{GLE-2}) is a homomorphism $\rho
:\pi_{1}\left(  E_{\tau}\backslash(E_{\tau}[2]\cup(\left \{  \pm p\right \}
+\Lambda_{\tau})),q_{0}\right)  \rightarrow SL(2,\mathbb{C})$, where
$q_{0}\not \in E_{\tau}[2]\cup(\left \{  \pm p\right \}  +\Lambda_{\tau})$ is a
base point. Let $\gamma_{\pm}\in \pi_{1}\left(  E_{\tau}\backslash(E_{\tau
}[2]\cup(\left \{  \pm p\right \}  +\Lambda_{\tau})),q_{0}\right)  $ be a simple
loop encircling $\pm p$ counterclockwise respectively, and $\ell_{j}$,
$j=1,2$, be two fundamental cycles of $E_{\tau}$ connecting $q_{0}$ with
$q_{0}+\omega_{j}$ such that $\ell_{j}$ does not intersect with $L+\Lambda
_{\tau}$ (here $L$ is the straight segment connecting $\pm p$) and satisfies%
\begin{equation}
\gamma_{+}\gamma_{-}=\ell_{1}\ell_{2}\ell_{1}^{-1}\ell_{2}^{-1}\text{ in }%
\pi_{1}\left(  E_{\tau}\backslash(\left \{  \pm p\right \}  +\Lambda_{\tau
}),q_{0}\right)  . \label{II-iv}%
\end{equation}
Since the local exponents of (\ref{GLE-2}) at $\pm p$ are $-\frac{1}{2}$ and
$\frac{3}{2}$ and $\pm p\not \in E_{\tau}[2]$ are apparent singularities, we
always have%
\begin{equation}
\rho(\gamma_{\pm})=-I_{2}. \label{89-2}%
\end{equation}
For any $k\in \{0,1,2,3\}$, the local exponents of GLE (\ref{GLE-2}) at
$\omega_{k}/2$ are $-n_{k}$ and $n_{k}+1$ with $n_{k}\in \mathbb{Z}$. Since the
potential of GLE (\ref{GLE-2}) is even elliptic, the local monodromy matrix of
GLE (\ref{GLE-2}) at $\omega_{k}/2$ is $I_{2}$ (see e.g. \cite[Lemma
2.2]{Takemura}). Therefore, the monodromy group of GLE (\ref{GLE-2}) is
generated by $\{-I_{2},\rho(\ell_{1}),\rho(\ell_{2})\}$. Together with
(\ref{II-iv}) and (\ref{89-2}), we immediately obtain $\rho(\ell_{1})\rho
(\ell_{2})=\rho(\ell_{2})\rho(\ell_{1})$, which implies that the monodromy
group of GLE (\ref{GLE-2}) is always \emph{abelian} and hence \emph{reducible}%
, i.e. \emph{all the monodromy matrices have at least a common eigenfunction}.
Clearly there are two cases:

\begin{itemize}
\item[(a)] Completely reducible, i.e. all the monodromy matrices have two
linearly independent common eigenfunctions. This case has been well studied in
\cite{CKL3}.

\item[(b)] Not completely reducible, i.e. the space of common eigenfunctions
is of dimension $1$: Up to a common conjugation,%
\begin{equation}
\rho(\ell_{1})=\varepsilon_{1}%
\begin{pmatrix}
1 & 0\\
1 & 1
\end{pmatrix}
,\text{ \  \  \ }\rho(\ell_{2})=\varepsilon_{2}%
\begin{pmatrix}
1 & 0\\
C & 1
\end{pmatrix}
, \label{Mono-2}%
\end{equation}
where $\varepsilon_{1},\varepsilon_{2}\in \{ \pm1\}$ and $C\in \mathbb{C}\cup \{
\infty \}$. Remark that if $C=\infty$, then (\ref{Mono-2}) should be understood
as%
\[
\rho(\ell_{1})=\varepsilon_{1}%
\begin{pmatrix}
1 & 0\\
0 & 1
\end{pmatrix}
,\text{ \  \  \ }\rho(\ell_{2})=\varepsilon_{2}%
\begin{pmatrix}
1 & 0\\
1 & 1
\end{pmatrix}
.
\]
In this case, $C$ is called the \emph{monodromy data} of GLE (\ref{GLE-2}).
\end{itemize}

Remark that for the projective ODE of GLE (\ref{GLE-2}) on $\mathbb{CP}^{1}$,
its monodromy representation is irreducible if and only if Case (a) occurs,
and reducible if and only if Case (b) occurs. Most of the references in the
literature are devoted to the case of irreducible representation on
$\mathbb{CP}^{1}$, but very few are devoted to studying reducible representation.

We are interested in \emph{the explicit formula of the monodromy data} $C$
when Case (b) occurs. This problem is important but challenging for general
$n_{k}$, which will be studied in a future paper. In this note, we focus on
the special case $(n_{0},n_{1},n_{2},n_{3})=(1,0,0,0)$, i.e. GLE%
\begin{equation}
y^{\prime \prime}(z)=\left[
\begin{array}
[c]{l}%
2\wp(z)+\frac{3}{4}(\wp(z+p)+\wp(z-p))\\
+A(\zeta(z+p)-\zeta(z-p))+B
\end{array}
\right]  y(z)\text{ \ in }E_{\tau}. \label{89-1}%
\end{equation}

Our first result is following.

\begin{theorem}
\label{thm1}Fix $\tau \in \mathbb{H}$ and $p\not \in E_{\tau}[2]$. If the
monodromy of GLE (\ref{89-1}) is not completely reducible, then the monodromy
data $C$ satisfies either%
\begin{equation}
\wp(p|\tau)=\frac{2g_{3}(C-\tau)^{3}-4(C\eta_{1}-\eta_{2})^{3}-g_{2}(C\eta
_{1}-\eta_{2})(C-\tau)^{2}}{(C-\tau)[12(C\eta_{1}-\eta_{2})^{2}-g_{2}%
(C-\tau)^{2}]}, \label{17-1}%
\end{equation}
with $(\varepsilon_{1},\varepsilon_{2})=(1,1)$ or%
\begin{equation}
\wp(p|\tau)=\frac{(\frac{g_{2}}{2}-3e_{k}^{2})(C\eta_{1}-\eta_{2})+\frac
{g_{2}}{4}e_{k}(C-\tau)}{3e_{k}(C\eta_{1}-\eta_{2})+(\frac{g_{2}}{2}%
-3e_{k}^{2})(C-\tau)}, \label{18-1}%
\end{equation}
for some $k\in \{1,2,3\}$ with%
\begin{equation}
(\varepsilon_{1},\varepsilon_{2})=\left \{
\begin{array}
[c]{l}%
(1,-1)\text{ \ if \ }k=1,\\
(-1,1)\text{ \ if \ }k=2,\\
(-1,-1)\text{ \ if \ }k=3.
\end{array}
\right.  \label{III-19}%
\end{equation}

\end{theorem}

Here $g_{2}=g_{2}(\tau)$ and $g_{3}=g_{3}(\tau)$ are the coefficients of%
\[
\wp^{\prime}(z|\tau)^{2}=4\wp(z|\tau)^{3}-g_{2}(\tau)\wp(z|\tau)-g_{3}%
(\tau)=4\prod_{k=1}^{3}(\wp(z|\tau)-e_{k}(\tau)).
\]
The formulas (\ref{17-1})-(\ref{18-1}) first appeared in \cite[(3.68)-(3.69)]%
{Takemura} without detailed proofs and later was obtained in \cite{CKL2}
independently, as explicit expressions of Riccati type solutions of
Painlev\'{e} VI equation. But their connection with the monodromy data seems
not be well addressed. Theorem \ref{thm1} can be proved directly without
applying Painlev\'{e} VI equation; see Section \ref{sect}.

Conversely, it is natural to consider the following \emph{Riemann-Hilbert
problem}: For fixed $\tau \in \mathbb{H}$ and $p\not \in E_{\tau}[2]$, and given
any $C$ satisfying (\ref{17-1}) or (\ref{18-1}), whether there exists GLE
(\ref{89-1}) (i.e. exist $A,B\in \mathbb{C}$) such that this $C$ is its
monodromy data? This problem is fundamental but seems not be settled. Our
second main result is to answer this question positively.

\begin{theorem}
\label{thm2}Fix $\tau \in \mathbb{H}$ and $p\not \in E_{\tau}[2]$.
\end{theorem}

\begin{itemize}
\item[(1)] If $C\in \mathbb{C}\cup \{ \infty \}$ satisfies the cubic equation
(\ref{17-1}), then there exists $A\in \mathbb{C}$ (and $B$ is given by $(p,A)$
via (\ref{i60}) with $(n_{0},n_{1},n_{2},n_{3})=(1,0,0,0)$) such that for the
correpsonding GLE (\ref{89-1}), up to a common conjugation,%
\begin{equation}
\rho(\ell_{1})=%
\begin{pmatrix}
1 & 0\\
1 & 1
\end{pmatrix}
,\text{ \  \  \ }\rho(\ell_{2})=%
\begin{pmatrix}
1 & 0\\
C & 1
\end{pmatrix}
. \label{17-3}%
\end{equation}

\item[(2)] Fix $k\in \{1,2,3\}$. If $C\in \mathbb{C}\cup \{ \infty \}$ satisfies
the equation (\ref{18-1}), then there exists $A\in \mathbb{C}$ such that for
the correpsonding GLE (\ref{89-1}), up to a common conjugation,%
\[
\rho(\ell_{1})=\varepsilon_{1}%
\begin{pmatrix}
1 & 0\\
1 & 1
\end{pmatrix}
,\text{ \  \  \ }\rho(\ell_{2})=\varepsilon_{2}%
\begin{pmatrix}
1 & 0\\
C & 1
\end{pmatrix}
,
\]
where $(\varepsilon_{1},\varepsilon_{2})$ is given by (\ref{III-19}).
\end{itemize}

If we know that the number of $A$'s with the monodromy of the corresponding
GLE (\ref{89-1}) being not completely reducible is $6$, then Theorem
\ref{thm2} might follow from Theorem \ref{thm1}. However, this assumption can
not hold for all $(p,\tau)$, and the proof of Theorem \ref{thm2} becomes
subtle. In Section 3, we will present a proof by applying the connection
between GLE (\ref{89-1}) and Painlev\'{e} VI equation. Theorem \ref{thm1} will
be proved in Section 2.

\section{Proof of Theorem \ref{thm1}}

\label{sect}

The purpose of this section is to prove Theorem \ref{thm1}. Let $y_{1}(z)$ be
a common eigenfunction. Then it is known (cf. \cite{CKL3}) that $y_{1}(z)$ can
be expressed as%
\begin{equation}
y_{1}(z)=e^{(r\eta_{1}+s\eta_{2})z}\frac{\sigma(z-a_{1})\sigma(z-a_{2}%
)}{\sigma(z)^{2}}\frac{\sigma(z)}{[\sigma(z-p)\sigma(z+p)]^{\frac{1}{2}}},
\label{y1}%
\end{equation}
where $(r,s)$ is determined by%
\begin{equation}
r+s\tau=a_{1}+a_{2},\text{ \ }r\eta_{1}+s\eta_{2}=\zeta(a_{1})+\zeta(a_{2}),
\label{89-6}%
\end{equation}
and $a_{1},a_{2}\not \in \Lambda_{\tau}$ satisfy
\begin{align}
&  \left[  \zeta(a_{i}+p)+\zeta(a_{i}-p)-2\zeta(a_{i})\right] \label{89-4}\\
=  &  2\left[  \zeta(a_{i}-a_{j})+\zeta(a_{j})-\zeta(a_{i})\right]  ,\text{
\ }\{i,j\}=\{1,2\}.\nonumber
\end{align}
Here $\sigma(z)=\sigma(z|\tau)$ is the Weierstrass sigma function defined by
$\frac{\sigma^{\prime}(z)}{\sigma(z)}:=\zeta(z)$. It is known that $\sigma(z)$
is an odd entire function with simple zeros only at the lattice points
$\Lambda_{\tau}$. Remark that (\ref{89-4}) can be easily obtained by inserting
(\ref{y1}) into GLE (\ref{89-1}).

\begin{proof}
[Proof of Theorem \ref{thm1}]Suppose the monodromy representation of GLE
(\ref{89-1}) is \emph{not} completely reducible, i.e. the space of common
eigenfunctions is of dimension $1$. Then $y_{1}(-z)=\pm y_{1}(z)$ (because
$y_{1}(-z)$ is also a common eigenfunction), i.e.%
\begin{equation}
\{a_{1},a_{2}\}=\{-a_{1},-a_{2}\} \text{ \ in \ }E_{\tau}. \label{89-7}%
\end{equation}
Remark that $\frac{\sigma(z)^{2}}{\sigma(z-p)\sigma(z+p)}$ is even elliptic.
Since $\ell_{j}$ does not intersect with $L+\Lambda_{\tau}$, the function
$\frac{\sigma(z)}{[\sigma(z-p)\sigma(z+p)]^{\frac{1}{2}}}$ is invariant under
analytic continuation along $\ell_{j}$ (see \cite{CKL3}), i.e.%
\begin{equation}
\ell_{j}^{\ast}\frac{\sigma(z)}{[\sigma(z-p)\sigma(z+p)]^{\frac{1}{2}}}%
=\frac{\sigma(z)}{[\sigma(z-p)\sigma(z+p)]^{\frac{1}{2}}},\text{ \ }j=1,2.
\label{89-8}%
\end{equation}
Here we denote by $\ell_{j}^{\ast}f(z)$ to be the function obtained from
$f(z)$ under analytic continuation along $\ell_{j}$. By (\ref{89-7}) there are
two cases.

\textbf{Case 1.} $a_{j}\not \in E_{\tau}[2]$ for $j=1,2$, then $a_{2}=-a_{1}$.

Then $(r,s)=(0,0)$, i.e.%
\[
y_{1}(z)=\frac{\sigma(z-a_{1})\sigma(z+a_{1})}{\sigma(z)^{2}}\frac{\sigma
(z)}{[\sigma(z-p)\sigma(z+p)]^{\frac{1}{2}}}.
\]
By (\ref{89-8}) and the transformation law%
\begin{equation}
\sigma(z+\omega_{j})=-e^{(z+{\omega_{j}}/{2})\eta_{j}}\sigma(z),\text{
\ }j=1,2, \label{i61}%
\end{equation}
it is easy to prove that%
\begin{equation}
\ell_{j}^{\ast}y_{1}(z)=y_{1}(z)\text{ \ for \ }j=1,2. \label{i67}%
\end{equation}
Therefore, $y_{1}(z)^{2}$ is even elliptic, i.e. up to a constant,%
\begin{equation}
y_{1}(z)^{-2}=\frac{\wp(z)-\wp(p)}{(\wp(z)-\wp(a_{1}))^{2}}. \label{i65}%
\end{equation}
On the other hand, it follows from the addition formulas%
\[
\zeta(u+v)+\zeta(u-v)-2\zeta(u)=\frac{\wp^{\prime}(u)}{\wp(u)-\wp(v)},
\]%
\[
\zeta(2u)-2\zeta(u)=\frac{\wp^{\prime \prime}(u)}{2\wp^{\prime}(u)},
\]
and (\ref{89-4}) that%
\[
\frac{\wp^{\prime}(a_{1})}{\wp(a_{1})-\wp(p)}=2\left[  \zeta(2a_{1}%
)-2\zeta(a_{i})\right]  =\frac{\wp^{\prime \prime}(a_{1})}{\wp^{\prime}(a_{1}%
)},
\]
i.e.%
\begin{equation}
\wp(p)=\wp(a_{1})-\frac{\wp^{\prime}(a_{1})^{2}}{\wp^{\prime \prime}(a_{1})}.
\label{iv}%
\end{equation}
From here and (\ref{i65}), it is easy to see that the \emph{residues} of
$y_{1}(z)^{-2}$ at $\pm a_{1}$ are both $0$, so%
\[
y_{1}(z)^{-2}=c_{3}\left[  \wp(z-a_{1})+\wp(z+a_{1})-2\wp(a_{1})\right]  ,
\]
where $c_{3}=\frac{\wp(a_{1})-\wp(p)}{\wp^{\prime}(a_{1})^{2}}$. Then%
\[
\chi(z):=\int_{0}^{z}\frac{1}{y_{1}(\xi)^{2}}d\xi=-c_{3}(\zeta(z-a_{1}%
)+\zeta(z+a_{1})+2\wp(a_{1})z),
\]
and so $\chi(z)$ is quasi-periodic with two quasi-periods:
\[
\chi_{1}=\chi(z+1)-\chi(z)=-2c_{3}(\eta_{1}+\wp(a_{1})),
\]%
\[
\chi_{2}=\chi(z+\tau)-\chi(z)=-2c_{3}(\eta_{2}+\wp(a_{1})\tau).
\]
Define%
\begin{equation}
C:=\frac{\chi_{2}}{\chi_{1}}=\frac{\eta_{2}+\wp(a_{1})\tau}{\eta_{1}+\wp
(a_{1})}, \label{i68}%
\end{equation}
and $y_{2}(z):=y_{1}(z)\chi(z)$. Since $y_{1}(z)$ is a solution of GLE
(\ref{89-1}) and $\chi(z)=\int_{0}^{z}\frac{1}{y_{1}(\xi)^{2}}d\xi$, it is
easy to see that $y_{2}(z)$ is a \emph{linearly independent} solution of GLE
(\ref{89-1}) with respect to $y_{1}(z)$. Recalling (\ref{i67}) and that
$y_{2}(z)$ is not a common eigenfunction, we know that $\chi_{1},\chi_{2}$ can
not vanish simultaneously. Consequently, if $\chi_{1}=0$, then $\chi_{2}%
\not =0$, $C=\infty$ and so%
\[
\ell_{1}^{\ast}%
\begin{pmatrix}
\chi_{2}y_{1}(z)\\
y_{2}(z)
\end{pmatrix}
=%
\begin{pmatrix}
\chi_{2}y_{1}(z)\\
y_{2}(z)
\end{pmatrix}
,\text{ \ i.e. }\rho(\ell_{1})=I_{2},
\]%
\[
\ell_{2}^{\ast}%
\begin{pmatrix}
\chi_{2}y_{1}(z)\\
y_{2}(z)
\end{pmatrix}
=%
\begin{pmatrix}
1 & 0\\
1 & 1
\end{pmatrix}%
\begin{pmatrix}
\chi_{2}y_{1}(z)\\
y_{2}(z)
\end{pmatrix}
,\text{ \ i.e. }\rho(\ell_{2})=%
\begin{pmatrix}
1 & 0\\
1 & 1
\end{pmatrix}
.
\]
If $\chi_{1}\not =0$, then $C\not =\infty$ and so%
\[
\ell_{1}^{\ast}%
\begin{pmatrix}
\chi_{1}y_{1}(z)\\
y_{2}(z)
\end{pmatrix}
=%
\begin{pmatrix}
1 & 0\\
1 & 1
\end{pmatrix}%
\begin{pmatrix}
\chi_{1}y_{1}(z)\\
y_{2}(z)
\end{pmatrix}
,\text{ \ i.e. }\rho(\ell_{1})=%
\begin{pmatrix}
1 & 0\\
1 & 1
\end{pmatrix}
,
\]%
\[
\ell_{2}^{\ast}%
\begin{pmatrix}
\chi_{1}y_{1}(z)\\
y_{2}(z)
\end{pmatrix}
=%
\begin{pmatrix}
1 & 0\\
C & 1
\end{pmatrix}%
\begin{pmatrix}
\chi_{1}y_{1}(z)\\
y_{2}(z)
\end{pmatrix}
,\text{ \ i.e. }\rho(\ell_{2})=%
\begin{pmatrix}
1 & 0\\
C & 1
\end{pmatrix}
.
\]
Clearly (\ref{i68}) gives%
\[
\wp(a_{1})=\frac{C\eta_{1}-\eta_{2}}{\tau-C}.
\]
Inserting this into (\ref{iv}) and using $\wp^{\prime}(a_{1})^{2}=4\wp
(a_{1})^{3}-g_{2}\wp(a_{1})-g_{3}$, $\wp^{\prime \prime}(a_{1})=6\wp(a_{1}%
)^{2}-g_{2}/2$, we easily obtain%
\[
\wp(p)=\frac{2g_{3}(C-\tau)^{3}-4(C\eta_{1}-\eta_{2})^{3}-g_{2}(C\eta_{1}%
-\eta_{2})(C-\tau)^{2}}{(C-\tau)[12(C\eta_{1}-\eta_{2})^{2}-g_{2}(C-\tau
)^{2}]}.
\]
This proves (\ref{17-1}) with $(\varepsilon_{1},\varepsilon_{2})=(1,1)$.

\textbf{Case 2.} $a_{l}\in E_{\tau}[2]$ for some $l\in \{1,2\}$, then
$\{a_{1},a_{2}\}=\{ \frac{\omega_{i}}{2},\frac{\omega_{j}}{2}\}$ for some
$i\not =j\in \{1,2,3\}$. Denote $\eta_{3}=\eta_{1}+\eta_{2}$.

Note from (\ref{89-6}) that $r\eta_{1}+s\eta_{2}=\frac{1}{2}(\eta_{i}+\eta
_{j})$. Let $\{k\}=\{1,2,3\}/\{i,j\}$. By (\ref{89-8}) and (\ref{i61}), it is
easy to prove that%
\[
\ell_{1}^{\ast}y_{1}(z)=\left(  -1\right)  ^{k+m\left(  k\right)  }%
y_{1}(z),\text{ \ }\ell_{2}^{\ast}y_{1}(z)=\left(  -1\right)  ^{k}y_{1}(z).
\]
Then $y_{1}(z)^{2}$ is even elliptic, i.e. up to a constant,%
\[
y_{1}(z)^{-2}=\frac{\wp(z)-\wp(p)}{(\wp(z)-e_{i})(\wp(z)-e_{j})}.
\]
Again a direct computation shows that the residue of $y_{1}(z)^{-2}$ at
$\frac{\omega_{i}}{2},\frac{\omega_{j}}{2}$ are both $0$, so%
\[
y_{1}(z)^{-2}=c_{3}\left(  \wp(z-\tfrac{\omega_{i}}{2})-e_{i}\right)
+c_{4}\left(  \wp(z-\tfrac{\omega_{j}}{2})-e_{j}\right)  ,
\]
where%
\[
c_{3}=\frac{2(e_{i}-\wp(p))}{(e_{i}-e_{j})\wp^{\prime \prime}(\frac{\omega_{i}%
}{2})},\text{ \ }c_{4}=\frac{2(e_{j}-\wp(p))}{(e_{j}-e_{i})\wp^{\prime \prime
}(\frac{\omega_{j}}{2})}.
\]
By $\wp^{\prime \prime}(\frac{\omega_{i}}{2})=6e_{i}^{2}-g_{2}/2$,
$g_{2}=4(e_{k}^{2}-e_{i}e_{j})$ and $e_{i}+e_{j}+e_{k}=0$, a direct
computation gives%
\begin{equation}
c_{3}+c_{4}=\frac{2g_{2}-12e_{k}^{2}-12e_{k}\wp(p)}{\wp^{\prime \prime}%
(\frac{\omega_{i}}{2})\wp^{\prime \prime}(\frac{\omega_{j}}{2})}, \label{iv1}%
\end{equation}%
\[
c_{3}e_{i}+c_{4}e_{j}=\frac{g_{2}e_{k}+12\wp(p)e_{k}^{2}-2g_{2}\wp(p)}%
{\wp^{\prime \prime}(\frac{\omega_{i}}{2})\wp^{\prime \prime}(\frac{\omega_{j}%
}{2})}.
\]
Define%
\[
\chi(z):=\int_{0}^{z}\frac{1}{y_{1}(\xi)^{2}}d\xi=-c_{3}\left(  \zeta
(z-\tfrac{\omega_{i}}{2})+e_{i}z\right)  -c_{4}\left(  \zeta(z-\tfrac
{\omega_{j}}{2})+e_{j}z\right)  ,
\]
then it follows from the Lengrede relation $\tau \eta_{1}-2\pi i=\eta_{2}$
that
\[
\chi_{1}=\chi(z+1)-\chi(z)=-\eta_{1}(c_{3}+c_{4})-(c_{3}e_{i}+c_{4}e_{j}),
\]%
\begin{align*}
\chi_{2}  &  =\chi(z+\tau)-\chi(z)=-\eta_{2}(c_{3}+c_{4})-(c_{3}e_{i}%
+c_{4}e_{j})\tau \\
&  =\tau \chi_{1}+2\pi i(c_{3}+c_{4}).
\end{align*}
Again we define%
\begin{equation}
C:=\frac{\chi_{2}}{\chi_{1}}=\tau+\frac{2\pi i(c_{3}+c_{4})}{\chi_{1}},
\label{iv5}%
\end{equation}
and $y_{2}(z):=y_{1}(z)\chi(z)$. Then as above, we have the following
conclusions: If $\chi_{1}=0$, then $\chi_{2}\not =0$, $C=\infty$ and so%
\[
\ell_{1}^{\ast}%
\begin{pmatrix}
\chi_{2}y_{1}(z)\\
y_{2}(z)
\end{pmatrix}
=\left(  -1\right)  ^{k+m\left(  k\right)  }%
\begin{pmatrix}
\chi_{2}y_{1}(z)\\
y_{2}(z)
\end{pmatrix}
,
\]%
\[
\ell_{2}^{\ast}%
\begin{pmatrix}
\chi_{2}y_{1}(z)\\
y_{2}(z)
\end{pmatrix}
=\left(  -1\right)  ^{k}%
\begin{pmatrix}
1 & 0\\
1 & 1
\end{pmatrix}%
\begin{pmatrix}
\chi_{2}y_{1}(z)\\
y_{2}(z)
\end{pmatrix}
,
\]%
\[
\text{i.e. }\rho(\ell_{1})=\left(  -1\right)  ^{k+m\left(  k\right)  }%
I_{2}\text{ \ and \ }\rho(\ell_{2})=\left(  -1\right)  ^{k}%
\begin{pmatrix}
1 & 0\\
1 & 1
\end{pmatrix}
.
\]
If $\chi_{1}\not =0$, then $C\not =\infty$ and so%
\[
\ell_{1}^{\ast}%
\begin{pmatrix}
\chi_{1}y_{1}(z)\\
y_{2}(z)
\end{pmatrix}
=\left(  -1\right)  ^{k+m\left(  k\right)  }%
\begin{pmatrix}
1 & 0\\
1 & 1
\end{pmatrix}%
\begin{pmatrix}
\chi_{1}y_{1}(z)\\
y_{2}(z)
\end{pmatrix}
,
\]%
\[
\ell_{2}^{\ast}%
\begin{pmatrix}
\chi_{1}y_{1}(z)\\
y_{2}(z)
\end{pmatrix}
=\left(  -1\right)  ^{k}%
\begin{pmatrix}
1 & 0\\
C & 1
\end{pmatrix}%
\begin{pmatrix}
\chi_{1}y_{1}(z)\\
y_{2}(z)
\end{pmatrix}
,
\]%
\[
\text{i.e. }\rho(\ell_{1})=\left(  -1\right)  ^{k+m\left(  k\right)  }%
\begin{pmatrix}
1 & 0\\
1 & 1
\end{pmatrix}
\text{ \ and \ }\rho(\ell_{2})=\left(  -1\right)  ^{k}%
\begin{pmatrix}
1 & 0\\
C & 1
\end{pmatrix}
.
\]
By (\ref{iv1})-(\ref{iv5}) we have
\[
\frac{C\eta_{1}-\eta_{2}}{\tau-C}=-\eta_{1}-\frac{\chi_{1}}{c_{3}+c_{4}}%
=\frac{c_{3}e_{i}+c_{4}e_{j}}{c_{3}+c_{4}}=\frac{(\frac{g_{2}}{2}-3e_{k}%
^{2})\wp(p)-\frac{g_{2}}{4}e_{k}}{3e_{k}\wp(p)+3e_{k}^{2}-\frac{g_{2}}{2}},
\]
and so%
\[
\wp(p)=\frac{(\frac{g_{2}}{2}-3e_{k}^{2})(C\eta_{1}-\eta_{2})+\frac{g_{2}}%
{4}e_{k}(C-\tau)}{3e_{k}(C\eta_{1}-\eta_{2})+(\frac{g_{2}}{2}-3e_{k}%
^{2})(C-\tau)}.
\]
This proves (\ref{18-1}) with $(\varepsilon_{1},\varepsilon_{2})$ given by
(\ref{III-19}).
\end{proof}

\section{Relation with Painlev\'{e} VI equation}

\label{RTS-PVI}

This section is devoted to the proof of Theorem \ref{thm2}. First we recall
the relation between GLE (\ref{GLE-2}) and Panlev\'{e} VI equation. The
well-known Painlev\'{e} VI equation with four parameters $(\alpha,\beta
,\gamma,\delta)$ (PVI$(\alpha,\beta,\gamma,\delta)$) is written as%
\begin{align}
\frac{d^{2}\lambda}{dt^{2}}=  &  \frac{1}{2}\left(  \frac{1}{\lambda}+\frac
{1}{\lambda-1}+\frac{1}{\lambda-t}\right)  \left(  \frac{d\lambda}{dt}\right)
^{2}-\left(  \frac{1}{t}+\frac{1}{t-1}+\frac{1}{\lambda-t}\right)
\frac{d\lambda}{dt}\nonumber \\
&  +\frac{\lambda(\lambda-1)(\lambda-t)}{t^{2}(t-1)^{2}}\left[  \alpha
+\beta \frac{t}{\lambda^{2}}+\gamma \frac{t-1}{(\lambda-1)^{2}}+\delta
\frac{t(t-1)}{(\lambda-t)^{2}}\right]  . \label{46-0}%
\end{align}
Due to its connection with many different disciplines in mathematics and
physics, PVI (\ref{46-0}) has been extensively studied in the past several
decades. See \cite{Hit1,GP,Lisovyy-Tykhyy,Okamoto1} and references therein.

One of the fundamental properties for PVI (\ref{46-0}) is the so-called
\emph{Painlev\'{e} property} which says that any solution $\lambda(t)$ of
(\ref{46-0}) has neither movable branch points nor movable essential
singularities; in other words, for any $t_{0}\in \mathbb{C}\backslash \{0,1\}$,
either $\lambda(t)$ is holomorphic at $t_{0}$ or $\lambda(t)$ has a pole at
$t_{0}$. Then it is reasonable to lift PVI (\ref{46-0}) to the covering space
$\mathbb{H=}\{ \tau|\operatorname{Im}\tau>0\}$ of $\mathbb{C}\backslash
\{0,1\}$ by the following transformation:%
\begin{equation}
t=\frac{e_{3}(\tau)-e_{1}(\tau)}{e_{2}(\tau)-e_{1}(\tau)},\text{ \ }%
\lambda(t)=\frac{\wp(p(\tau)|\tau)-e_{1}(\tau)}{e_{2}(\tau)-e_{1}(\tau)}.
\label{II-130}%
\end{equation}
Consequently, $p(\tau)$ satisfies the following \emph{elliptic form} of PVI
(cf. \cite{Babich-Bordag,Y.Manin}):
\begin{equation}
\frac{d^{2}p(\tau)}{d\tau^{2}}=\frac{-1}{4\pi^{2}}\sum_{k=0}^{3}\alpha_{k}%
\wp^{\prime}\left(  \left.  p(\tau)+\tfrac{\omega_{k}}{2}\right \vert
\tau \right)  , \label{124-0}%
\end{equation}
with parameters $\left(  \alpha_{0},\alpha_{1},\alpha_{2},\alpha_{3}\right)
=(\alpha,-\beta,\gamma,\tfrac{1}{2}-\delta)$. The Painlev\'{e} property of PVI
(\ref{46-0}) implies that function $\wp(p(\tau)|\tau)$ is a single-valued
meromorphic function in $\mathbb{H}$.

In this note, we only consider the special case $\alpha_{0}=\frac{9}{8}$ and
$\alpha_{k}=\frac{1}{8}$ for $k\geq1$, i.e.%
\begin{equation}
\frac{d^{2}p(\tau)}{d\tau^{2}}=\frac{-1}{4\pi^{2}}\left(  \tfrac{9}{8}%
\wp^{\prime}\left(  \left.  p(\tau)\right \vert \tau \right)  +\sum_{k=1}%
^{3}\tfrac{1}{8}\wp^{\prime}\left(  \left.  p(\tau)+\tfrac{\omega_{k}}%
{2}\right \vert \tau \right)  \right)  , \label{124-5}%
\end{equation}
which is the elliptic form of PVI$(\tfrac{9}{8},\tfrac{-1}{8},\tfrac{1}%
{8},\tfrac{3}{8})$.

It is well known (cf. \cite{Okamoto1,Takemura,CKL2}) that any solution of
PVI$(\tfrac{9}{8},\tfrac{-1}{8},\tfrac{1}{8},\tfrac{3}{8})$ can be obtained
from that of PVI$(\tfrac{1}{8},\tfrac{-1}{8},\tfrac{1}{8},\tfrac{3}{8})$ (i.e.
the case studied by Hitchin \cite{Hit1}) by the Okamoto transformation
\cite{Okamoto1}. By this way, the following result gives explicit expressions
of some solutions of the elliptic from (\ref{124-5}).\medskip

\noindent \textbf{Theorem A. }\cite{Takemura,CKL2} \emph{For any }%
$k\in \{0,1,2,3\}$\emph{ and} $C\in \mathbb{C}\cup \{ \infty \}$, $p_{C}%
^{(k)}(\tau)$\  \emph{is a solution of the elliptic form (\ref{124-5}), where
}$\wp(p_{C}^{(k)}(\tau)|\tau)$\emph{ are given by:}{\allowdisplaybreaks%
\begin{align}
&  \wp(p_{C}^{(0)}(\tau)|\tau)=\label{III-17}\\
&  \frac{2g_{3}(\tau)(C-\tau)^{3}-4(C\eta_{1}(\tau)-\eta_{2}(\tau))^{3}%
-g_{2}(\tau)(C\eta_{1}(\tau)-\eta_{2}(\tau))(C-\tau)^{2}}{(C-\tau
)[12(C\eta_{1}(\tau)-\eta_{2}(\tau))^{2}-g_{2}(\tau)(C-\tau)^{2}]},\nonumber
\end{align}
} \emph{and}%
\begin{equation}
\wp(p_{C}^{(k)}(\tau)|\tau)=\frac{(\frac{g_{2}(\tau)}{2}-3e_{k}(\tau
)^{2})(C\eta_{1}(\tau)-\eta_{2}(\tau))+\frac{g_{2}(\tau)}{4}e_{k}(\tau
)(C-\tau)}{3e_{k}(\tau)(C\eta_{1}(\tau)-\eta_{2}(\tau))+(\frac{g_{2}(\tau)}%
{2}-3e_{k}(\tau)^{2})(C-\tau)} \label{III-18}%
\end{equation}
\emph{with} $k\in \{1,2,3\}$. \emph{Equivalently,}%
\[
\lambda_{C}^{(k)}(t):=\frac{\wp(p_{C}^{(k)}(\tau)|\tau)-e_{1}(\tau)}%
{e_{2}(\tau)-e_{1}(\tau)},\text{ }t=\frac{e_{3}(\tau)-e_{1}(\tau)}{e_{2}%
(\tau)-e_{1}(\tau)}%
\]
\emph{is a Riccati type solution of PVI}$(\frac{9}{8},\frac{-1}{8},\frac{1}%
{8},\frac{3}{8})$.\medskip

We recall \cite{Takemura,Chen-Kuo-Lin} that the elliptic form (\ref{124-5})
governs the isomonodromic deformation of GLE (\ref{89-1}). More
precisely,$\medskip$

\noindent \textbf{Theorem B.} \cite{Chen-Kuo-Lin}$\ p(\tau)$ \emph{is a
solution of (\ref{124-5}) if and only if there exists $A(\tau)$ (Consequently,
$B(\tau)$ is determined by $(p(\tau),A(\tau))$ via (\ref{i60}) with }%
$(n_{0},n_{1},n_{2},n_{3})=(1,0,0,0)$\emph{) such that GLE (\ref{89-1}) with
$(p,A,B,\tau)=(p(\tau),A(\tau),B(\tau),\tau)$ preserves the monodromy as
$\tau$ deforms}.\medskip

We call such GLE (\ref{89-1}) with $(p,A,B)=(p(\tau)$, $A(\tau)$, $B(\tau))$
to be the corresponding GLE of the solution $p(\tau)$.

\begin{remark}
\label{rem}The formulas (\ref{III-17})-(\ref{III-18}) were derived in
\cite{Takemura,Chen-Kuo-Lin} but the monodromy of the corresponding GLE
(\ref{89-1}) was not thoroughly discussed. In \cite{CKL3}, it was proved that
for $p_{C}^{(k)}(\tau)$, the monodromy of the corresponding GLE (\ref{89-1})
is not completely reducible (i.e. the form of (\ref{Mono-2})) but without
proving that the monodromy data is precisely the same $C$. This assertion,
which is strongly suggested by Theorem \ref{thm1}, will be proved in following.
\end{remark}

\begin{proof}
[Proof of Theorem \ref{thm2}]Fix $\tau_{0}\in \mathbb{H}$ and $p_{0}%
\not \in E_{\tau_{0}}[2]$.

Suppose $C\in \mathbb{C}\cup \{ \infty \}$ satisfies the cubic equation
(\ref{17-1}) with $\tau=\tau_{0}$ and $p=p_{0}$. Our goal is to prove the
existence of $A_{0}\in \mathbb{C}$ such that (\ref{17-3}) holds for the GLE
(\ref{89-1}) with $(p,A,\tau)=(p_{0},A_{0},\tau_{0})$.

Note that $\wp(p_{0}|\tau_{0})=\wp(p_{C}^{(0)}(\tau_{0})|\tau_{0})$, where
$p_{C}^{(0)}(\tau)$ is a solution of the elliptic form (\ref{124-5}) given in
Theorem A, so without loss of generality, we may assume $p_{0}=p_{C}%
^{(0)}(\tau_{0})$ in $E_{\tau_{0}}$. Then by Theorem B and Remark \ref{rem},
there is $A(\tau)$ such that the corresponding GLE (\ref{89-1}) with
$(p,A,\tau)=(p_{C}^{(0)}(\tau),A(\tau),$ $\tau)$ is monodromy preserving with
\textit{not} completely monodromy representation as long as $p_{C}^{(0)}%
(\tau)\not \in E_{\tau}[2]$. Fix any $\tau \in \mathbb{H}$ such that
$p_{C}^{(0)}(\tau)\not \in E_{\tau}[2]$. For GLE (\ref{89-1}) with
$(p,A,\tau)=(p_{C}^{(0)}(\tau),A(\tau),\tau)$, by repeating the proof of
Theorem \ref{thm1}, there exists a monodromy data $C_{1}=C_{1}(\tau)$ such
that either
\begin{align}
&  \wp(p_{C}^{(0)}(\tau)|\tau)=\label{10}\\
&  \frac{2g_{3}(\tau)(C_{1}-\tau)^{3}-4(C_{1}\eta_{1}(\tau)-\eta_{2}%
(\tau))^{3}-g_{2}(\tau)(C_{1}\eta_{1}(\tau)-\eta_{2}(\tau))(C_{1}-\tau)^{2}%
}{(C_{1}-\tau)[12(C_{1}\eta_{1}(\tau)-\eta_{2}(\tau))^{2}-g_{2}(\tau
)(C_{1}-\tau)^{2}]}\nonumber
\end{align}
or%
\begin{equation}
\wp(p_{C}^{(0)}(\tau)|\tau)=\frac{(\frac{g_{2}(\tau)}{2}-3e_{k}(\tau
)^{2})(C_{1}\eta_{1}(\tau)-\eta_{2}(\tau))+\frac{g_{2}(\tau)}{4}e_{k}%
(\tau)(C_{1}-\tau)}{3e_{k}(\tau)(C_{1}\eta_{1}(\tau)-\eta_{2}(\tau
))+(\frac{g_{2}(\tau)}{2}-3e_{k}(\tau)^{2})(C_{1}-\tau)} \label{11}%
\end{equation}
with some $k\in \{1,2,3\}$ holds for any $\tau$ satisfying $p_{C}^{(0)}%
(\tau)\not \in E_{\tau}[2]$. Since GLE (\ref{89-1}) with $(p,A,\tau
)=(p_{C}^{(0)}(\tau),A(\tau),\tau)$ is monodromy preserving, so the monodromy
data $C_{1}$ is a constant dependent of $\tau$. By comparing the RHS of
(\ref{III-17}) and (\ref{10})-(\ref{11}) which hold for any $\tau$ satisfying
$p_{C}^{(0)}(\tau)\not \in E_{\tau}[2]$ (i.e. \textit{any }$\tau \in \mathbb{H}%
$\textit{ except a discrete set}), we easily\ conclude that (\ref{10}) and
$C_{1}=C$ hold. Then by letting $A_{0}=A(\tau_{0})$ and using $p_{0}%
=p_{C}^{(0)}(\tau_{0})$, we see that $C$ is precisely the monodromy data of
GLE (\ref{89-1}) with $(p,A,\tau)=(p_{0},A_{0},\tau_{0})$ and $(\varepsilon
_{1},\varepsilon_{2})=(1,1)$. This proves (1).

The assertion (2) can be proved similarly.
\end{proof}

The above proof highlights the effectiveness of Painlev\'{e} VI equation to solve such kind of Riemman-Hilbert problems. We will apply this idea to consider the general case $n_k\in\mathbb{Z}$ for all $k$ in a future work.


\begin{thebibliography}{99}                                                                                               %


\bibitem {Babich-Bordag}{ M. V. Babich and L. A. Bordag; \textit{The elliptic
form of the sixth Painlev\'{e} equation}. Preprint NT Z25/1997, Leipzig
(1997).}

\bibitem {Chen-Kuo-Lin}Z. Chen, T.J. Kuo and C.S. Lin; \textit{Hamiltonian
system for the elliptic form of Painlev\'{e} VI equation}. J. Math. Pures
Appl. \textbf{106} (2016), 546-581.

\bibitem {CKL2}Z. Chen, T.J. Kuo and C.S. Lin; \textit{Simple zero property of
some holomorphic functions on the moduli space of tori}. preprint, 2016.

\bibitem {CKL3}Z. Chen, T.J. Kuo and C.S. Lin; \textit{Painlev\'{e} VI
equation, modular forms and application}. preprint, 2016.

\bibitem {Hit1}{ N. J. Hitchin; \textit{Twistor spaces, Einstein metrics and
isomonodromic deformations}. J. Differ. Geom. \textbf{42} (1995), no.1,
30-112. }

\bibitem {GP}{ K. Iwasaki, H. Kimura, S. Shimomura and M. Yoshida;
\textit{From Gauss to Painlev\'{e}: A Modern Theory of Special Functions}.
Springer vol. E16, 1991.}

\bibitem {Lisovyy-Tykhyy}{ O. Lisovyy and Y. Tykhyy; \textit{Algebraic
solutions of the sixth Painlev\'{e} equation}. J. Geom. Phys. \textbf{85}
(2014), 124-163. }

\bibitem {Y.Manin}{ Y. Manin; \textit{Sixth Painlev\'{e} quation, universal
elliptic curve, and mirror of} $\mathbb{P}^{2}$. Amer. Math. Soc. Transl. (2),
\textbf{186} (1998), 131--151.}

\bibitem {Okamoto1}{ K. Okamoto; \textit{Studies on the Painlev\'{e}
equations. I. Sixth Painlev\'{e} equation} $P_{VI}$. Ann. Mat. Pura Appl.
\textbf{146} (1986), 337-381.}

\bibitem {Takemura}K. Takemura; \textit{The Hermite-Krichever Ansatz for
Fuchsian equations with applications to the sixth Painlev\'{e} equation and to
finite gap potentials}. Math. Z. \textbf{263} (2009), 149-194.
\end{thebibliography}
\end{document}